\documentclass[a4paper,12pt]{article}

\usepackage[table,dvipsnames,svgnames,x11names,hyperref]{xcolor}
\usepackage[utf8]{inputenc}

\usepackage{lmodern}
\usepackage[T1]{fontenc}

\usepackage[top=1in, bottom=1.25in, left=1.25in, right=1.25in]{geometry}
\usepackage{amsmath,amssymb,amsthm}

\usepackage{tikz}
\usepackage{array}
\usepackage{enumerate}
\usepackage{graphicx}
\usepackage{float}
\usepackage{caption}
\usepackage{subcaption}
\usepackage[colorlinks=true, allcolors=MidnightBlue]{hyperref}

\newtheorem{thm}{Theorem}[section]
\newtheorem{cor}[thm]{Corollary}
\newtheorem{lem}[thm]{Lemma}
\newtheorem{prop}[thm]{Proposition}

\theoremstyle{definition}

\theoremstyle{remark}

\newtheorem*{rem}{Remark}
\newtheorem*{remarks}{Remarks}
\newtheorem*{example}{Example}
\newtheorem*{examples}{Examples}

\newcommand\leftidx[3]{%
  {\vphantom{#2}}\hspace{.08em}#1\hspace{-.08em}#2#3%
}

\newcommand{\overbar}[1]{\mkern 1.5mu\overline{\mkern-1.5mu#1\mkern-1.5mu}\mkern 1.5mu}

\newcommand{\divides}{\bigm|}
\newcommand{\ndivides}{%
  \mathrel{\mkern.5mu 
    \ooalign{\hidewidth$\big|$\hidewidth\cr$\nmid$\cr}%
  }%
}

\author{Sel\c{c}uk Kayacan\thanks{This work is supported by the T{\"U}B\.{I}TAK 2214/A Grant Program: 1059B141401085.}}
\title{Dominating Sets in Intersection Graphs\\ of Finite Groups}
\date{}

\begin{document}

\maketitle

\small

\begin{center}
  TÜBİTAK BİLGEM\\
  41470 Gebze, Kocaeli, Turkey. \\
  {\it e-mail:} \href{mailto:kayacan.selcuk@gmail.com}{kayacan.selcuk@gmail.com}
\end{center}

     \begin{abstract}
     Let $G$ be a group. The intersection graph $\Gamma(G)$ of $G$ is an undirected graph without loops and multiple edges defined as follows: the vertex set is the set of all proper non-trivial subgroups of $G$, and there is an edge between two distinct vertices $H$ and $K$ if and only if $H\cap K \neq 1$ where $1$ denotes the trivial subgroup of $G$. In this paper we studied the dominating sets in intersection graphs of finite groups. It turns out a subset of the vertex set is a dominating set if and only if the union of the corresponding subgroups contains the union of all minimal subgroups. We classified abelian groups by their domination number and find upper bounds for some specific classes of groups. Subgroup intersection is related with Burnside rings. We introduce the notion of intersection graph of a $G$-set (somewhat generalizing the ordinary definition of intersection graph of a group) and establish a general upper bound for the domination number of $\Gamma(G)$ in terms of subgroups satisfying a certain property in Burnside ring. Intersection graph of $G$ is the $1$-skeleton of the simplicial complex whose faces are the sets of proper subgroups which intersect non-trivially. We call this simplicial complex intersection complex of $G$ and show that it shares the same homotopy type with the order complex of proper non-trivial subgroups of $G$. We also prove that if domination number of $\Gamma(G)$ is $1$, then intersection complex of $G$ is contractible.

  \smallskip
  \noindent 2010 {\it Mathematics Subject Classification.} Primary:
  20D99; Secondary: 05C69, 05C25, 20C05, 55U10.

  \smallskip
  \noindent Keywords: Finite groups; subgroup; intersection graph;
  dominating sets; domination number; Burnside ring; order complex

\end{abstract}

     \section{Introduction}
\label{sec:intro}

Let $\mathcal{F}$ be the set of proper subobjects of an object with an algebraic structure. In \cite{Yaraneri2013} the \emph{intersection graph of $\mathcal{F}$} is defined in the following way: there is a vertex for each subobject in $\mathcal{F}$ other than the zero object, where the zero object is the object having a unique endomorphism, and there is an edge between two vertices whenever the intersection of the subobjects representing the vertices is not the zero object. In particular, if $\mathcal{F}$ is the set of proper subgroups of a group $G$, then the zero object is the trivial subgroup. The intersection graph of (the proper subgroups of) $G$ will be denoted by $\Gamma(G)$.

Intersection graphs first defined for semigroups by Bos\'ak in \cite{Bosak1963}. Let $S$ be a semigroup. The intersection graph of the semigroup $S$ is defined in the following way: the vertex set is the set of proper subsemigroups of $S$ and there is an edge between two distinct vertices $A$ and $B$ if and only if $A\cap B\neq \varnothing$. Interestingly, this definition is not in the scope of the abstract generalization given in the preceding paragraph. Afterwards, in \cite{CsakanyPollak1969} Cs\'ak\'any and Poll\'ak adapted this definition into groups in the usual way. Still there are analogous definitions. For example, in \cite{ChakrabartyEtAl2009} authors studied the intersection graphs of ideals of a ring. In particular, they determine the values of $n$ for which the intersection graph of the ideals of $\mathbb{Z}_n$ is connected, complete, bipartite, planar or has a cycle. For the corresponding literature the reader may also refer to \cite{JafariRad2010,JafariRad2011,LaisonQing2010,Shen2010,Zelinka1975} and some of the references therein.

As is well-known subgroups of a group $G$ form a lattice $L(G)$ ordered by set inclusion. Some of the structural properties of a group may be inferred by studying its subgroup lattice. Intersection graphs of groups are natural objects and are intimately related with subgroup lattices. Actually, given the subgroup lattice one can recover the intersection graph but not vice versa in general. Intuitively, by passing from $L(G)$ to $\Gamma(G)$ a certain amount of knowledge should be lost. In \cite{KayacanEtAl2015b} authors show that finite abelian groups can almost be distinguished by their intersection graphs. The same result was proven previously for subgroup lattices in \cite{Baer1939} (see also \cite[Corollary~1.2.8]{Schmidt1994}). Therefore, rather surprisingly subgroup lattices and intersection graphs holds the same amount of information on the subgroup structure if the group is abelian.

By defining intersection graphs we attach a graph to a group, like in the case of Cayley graphs. So, there are two natural directions we may follow. First, we may study the graph theoretical properties of intersection graphs by means of group theoretical arguments. This is straightforward. For example we may ask for which groups their intersection graphs are planar \cite{KayacanEtAl2015a,AhmediEtAl2016} or connected \cite{Lucido2003,Shen2010,Kayacan2018}. And second, we may study the algebraic properties of groups by means of combinatorial arguments applied to the intersection graphs though this part seems to require more ingenuity.

In this paper, we study the dominating sets in intersection graphs. A \emph{dominating set} $D$ of a graph $\Gamma$ is a subset of the vertex set $V$ such that any vertex not in $D$ is adjacent to some vertex in $D$. The \emph{domination number} $\gamma(\Gamma)$ of $\Gamma$ is the smallest cardinal of the dominating sets for $\Gamma$. Vizing's conjecture from 1968 asserts that for any two graphs $\Gamma$ and $\Gamma'$ the product $\gamma(\Gamma)\gamma(\Gamma')$ is at most the domination number of the Cartesian product of $\Gamma$ and $\Gamma'$. Despite the efforts of many mathematicians this conjecture is still open (see \cite{BresarEtAl2012}). Given a graph $\Gamma$ and an integer $n$ the dominating set problem asks whether is there a dominating vertex set of size at most $n$. It is a classical instance of a $NP$-complete decision problem. There are many papers on domination theory covering algorithmic aspects as well. More can be found on this subject for example in \cite{BresarEtAl2007,FominEtAl2005,HaynesEtAl1997} and references therein.

It is easy to observe that a subset $\mathcal{D}$ of the vertex set $V(G)$ of the intersection graph of the group $G$ is a dominating set if and only if for any minimal subgroup $A$ of $G$ there exists a $H\in\mathcal{D}$ such that $A\leq H$. This in turn implies that $\mathcal{D}$ is a dominating set if and only if the union of the subgroups in $\mathcal{D}$ contains all minimal subgroups of $G$ as a subset, or equivalently, if and only if the union of the subgroups in $\mathcal{D}$ contains all elements of $G$ of prime order. In particular, the set of minimal subgroups and the set of maximal subgroups are dominating sets. We denote the domination number of $\Gamma(G)$ simply by $\gamma(G)$ and call this invariant of the group the \emph{domination number of $G$}. Let $G$ be a finite group. We shall note that a dominating set $\mathcal{D}$ of minimal size might be assumed to consist of maximal subgroups as any proper subgroup of a finite group is contained in a maximal subgroup. In particular, there exists a dominating set $\mathcal{D}$ such that each element of $\mathcal{D}$ is a maximal subgroup and the cardinality of $\mathcal{D}$ is $\gamma(G)$.

In \cite{Cohn1994}, Cohn defined a group as an $n$-sum group if it can be written as the union of $n$ of its proper subgroups and of no smaller number. Let $G$ be an $n$-sum group. In the light of the previous paragraphs it is clear that $\gamma(G)\leq n$. Notice that any non-trivial finite group can be written as the union of its proper subgroups unless it is cyclic, hence it is reasonable to call the cyclic group $C_m$ of order $m$ an $\aleph_0$-sum group. On the other hand a cyclic group of order $p^s$ with $p$ a prime contains a unique minimal subgroup, therefore $\gamma(C_{p^s})=1$ provided $s>1$. The intersection graph $\Gamma(G)$ is the empty graph whenever $G$ is trivial or isomorphic to a cyclic group of prime order and in such case we adopt the convention $\gamma(G)=\aleph_0$. The reason for this will be justified when we prove Lemma~\ref{lem:quotient}. Let $G$ be an $n$-sum group. It can be easily observed that $\gamma(G)=n$ if $G$ is isomorphic to the one of the following groups
$$ C_p,\qquad C_p\times C_p,\qquad C_p\times C_q,\qquad C_p\rtimes C_q, $$
where $p$ and $q$ are some distinct prime numbers. The reader may refer to for example \cite{Cohn1994,DetomiEtAl2008,GaronziEtAl2015} for the literature on $n$-sum groups.

Classifying groups by their domination number is a difficult problem. Even the determination of groups with domination number $1$ seems to be intractable. In Sections~\ref{sec:abelian},~\ref{sec:solvable}, and~\ref{sec:permutation} we determine upper bounds for the domination number of particular classes of groups. For example, abelian groups can be classified by their domination number (see Theorem~\ref{thm:abelian}) and the domination number of a supersolvable group is at most $p+1$ for some prime divisor $p$ of its order (see Proposition~\ref{prop:super}). It turns out symmetric groups forms an interesting class in our context. In Section~\ref{sec:permutation} we find some upper bounds for the symmetric groups by their degree (see Theorem~\ref{thm:Sn}) and show that those bounds are applicable also for the primitive subgroups containing an odd permutation (see Corollary~\ref{cor:primitive}).

In Section~\ref{sec:G-sets} we introduce \emph{intersection graphs of $G$-sets}. This notion in a sense generalizes the ordinary definition of intersection graphs of groups (see Proposition~\ref{prop:sigma}). Subgroup intersection is related with the multiplication operator of Burnside rings and the ultimate aim in this section is to incorporate the Burnside ring context into our discussion. We show that the domination number $\gamma(G)$ can be bounded by the sum of the indices of the normalizers of some subgroups in $G$ satisfying a certain property as a collection in the Burnside ring (see Proposition~\ref{prop:index}).

There is an extensive literature on combinatorial objects associated with algebraic structures. An alternative path in this direction is to introduce order complexes of subgroups and thereby rendering the use of topological terms possible (see, for example \cite{Brown1975,Quillen1978,HawkesEtAl1989,Smith2011}). Similar work using subgroup lattices, frames, coset posets, and quandles has also appeared in literature (see, \cite{ShareshianEtAl2012,ShareshianEtAl2016,Fumagalli2009,HeckenbergerEtAl2015}).

A natural construction in which a simplicial complex $K(G)$ is associated with a group $G$ is the following: the underlying set of $K(G)$ is the vertex set of $\Gamma(G)$ and for each vertex $H$ in $\Gamma(G)$ there is an associated simplex $\sigma_H$ in $K(G)$ which is defined as the set of proper subgroups of $G$ containing $H$. Clearly, the common face of $\sigma_H$ and $\sigma_K$ is $\sigma_{\langle H,K\rangle}$. Alternatively, $K(G)$ is the simplicial complex whose faces are the sets of proper subgroups of $G$ which intersect non-trivially. Observe that $\Gamma(G)$ is the $1$-skeleton of $K(G)$. By an argument due to Volkmar Welker, $K(G)$ shares the same homotopy type with the order complex of proper non-trivial subgroups of $G$ (see Proposition~\ref{prop:homotopy}).  In Section~\ref{sec:complex} we study the \emph{intersection complex $K(G)$} and prove that if domination number of $\Gamma(G)$ is $1$, then intersection complex of $G$ is contractible (see Corollary~\ref{cor:dom1}). 

\section{Preliminaries}
\label{sec:pre}

First we recall some of the basic facts from standard group theory.

\begin{remarks}
\mbox{}
  \begin{enumerate} 
  \item \label{rem1} (Product Formula, see {\cite[Theorem~2.20]{Rotman1995}}) 
    $|XY||X\cap Y|=|X||Y|$ for any two subgroups $X$ and $Y$ of a finite group.
  \item \label{rem2} (Sylow's Theorem, see {\cite[Theorem~4.12]{Rotman1995}})
    \begin{enumerate}[(i)]
      \item If $P$ is a Sylow $p$-subgroup of a finite group $G$, then all Sylow $p$-subgroups of $G$ are conjugate to $P$.
      \item If there are $r$ Sylow $p$-subgroups, then $r$ is a divisor of $|G|$ and $r \equiv 1 \pmod{p}$.
    \end{enumerate}
  \item \label{rem3} (Hall's Theorem, see {\cite[Theorem~4.1]{Gorenstein1980}}) If $G$ is a finite solvable group, then any $\pi$-subgroup is contained in a Hall $\pi$-subgroup. Moreover, any two Hall $\pi$-subgroups are conjugate.
  \item \label{rem4} (Correspondence Theorem, see {\cite[Theorem~2.28]{Rotman1995}}) Let $N\trianglelefteq G$ and let $\nu\colon G\to G/N$ be the canonical morphism. Then $S\mapsto \nu(S)=S/N$ is a bijection from the family of all those subgroups $S$ of $G$ which contain $N$ to the family of all the subgroups of $G/N$.
  \item \label{rem5} (Dedekind's Lemma, see {\cite[Exercise~2.49]{Rotman1995}}) Let $H,K,$ and $L$ be subgroups of $G$ with $H\leq L$. Then $HK\cap L=H(K\cap L)$.  
  \end{enumerate}
\end{remarks}

Let $G$ be a finite group. We denote by $N_G$ the subgroup of $G$ generated by its minimal subgroups. Obviously, $N_G$ is a characteristic subgroup. If $G\cong C_p$ with $p$ a prime, then we take $N_G=G$. Adapting the module theoretical parlance we might call a subgroup of a group \emph{essential} provided that it contains all minimal subgroups. Thus $N_G$ is the smallest essential subgroup. Notice that if $G$ is abelian and $G\ncong C_p$, then $N_G$ is the \emph{socle} of $G$.

\begin{lem}\label{lem:dom1}
  For a finite group $G$, the following statements are equivalent.
  \begin{enumerate}[(i)]
  \item The domination number of $\Gamma(G)$ is $1$.
  \item $N_G$ is a proper (normal) subgroup of $G$.
  \item $G$ is a (non-split) extension of $N_G$ by a non-trivial group $H$.
  \end{enumerate}
\end{lem}

\begin{proof}
  (i)\!$\implies$\!(ii): There must be a proper subgroup $H$ of $G$ such that $H\cap K\neq 1$ for any non-trivial subgroup $K$ of $G$. In particular, $H$ intersects non-trivially, and hence contains, any minimal subgroup in $N_G$. That is, $H\geq N_G$. However, $H$ is a proper subgroup of $G$, so is $N_G$.

(ii)\!$\implies$\!(iii): Since $N_G$ is a proper normal subgroup, $G$ is an extension of $N_G$ by a non-trivial group $U$. Notice that this extension cannot split, as otherwise, there would be a subgroup of $G$ isomorphic to $H$ which intersects $N_G$ trivially. However, this contradicts with the definition of $N_G$.

(iii)\!$\implies$\!(i): Clearly $N_G$ is a proper subgroup intersecting any subgroup non-trivially; hence $\{N_G\}$ is a dominating set for $\Gamma(G)$.
\end{proof}

\begin{cor}
  If $G$ is a finite simple group, then $\gamma(G)>1$. \qed
\end{cor}

In general, there is \emph{no} relation between the domination number of a group and its subgroups. As a simple example consider the dihedral group $D_8=\langle a,b\divides  a^4=b^2=1, bab=a^3\rangle$ of order $8$. It has three maximal subgroups $\langle a^2,b\rangle$, $\langle a^2, ab\rangle$, $\langle a\rangle$ and the first two of them together dominate $\Gamma(D_8)$. Moreover, as $D_8=\langle ab, b\rangle$, we have $\gamma(D_8)=2$ by Lemma~\ref{lem:dom1}. On the other hand, $\langle a^2,b\rangle\cong C_2\times C_2$ and since its intersection graph consists of isolated vertices, we have $\gamma(\langle a^2,b\rangle)=3$; whereas $\gamma(\langle a\rangle) = 1$ as it is isomorphic to the cyclic group of order four. However, by imposing some conditions on the subgroup $H$ of $G$, it is easy to prove that $\gamma(H)\leq \gamma(G)$ holds.

\begin{lem}\label{lem:subgroup}
  Let $H$ be a subgroup of $G$ and $\mathcal{D}$ be a dominating set of $\Gamma(G)$. Then $\gamma(H)\leq |\mathcal{D}|$ provided that none of the elements of $\mathcal{D}$ contains $H$. In particular, $\gamma(H)\leq\gamma(G)$ if there is such a dominating set with cardinality $\gamma(G)$. 
\end{lem}

\begin{proof}
  Observe that if none of the elements of $\mathcal{D}$ contains $H$, then the set $\mathcal{D}_H:=\{D\cap H\,\colon\, D\in\mathcal{D}\}$ is a dominating set for $\Gamma(H)$.
\end{proof}

The following result will be very useful in our later arguments.

\begin{lem}\label{lem:quotient}
  Let $N$ be a normal subgroup of $G$. If $G/N\ncong C_p$, then $\gamma(G)\leq \gamma(G/N)$. Moreover, the condition $G/N\ncong C_p$ can be removed if $G$ is a finite group.
\end{lem}

\begin{proof}
Let $\mathcal{D}$ be a dominating set for $\Gamma(G/N)$ and set $\overbar{\mathcal{D}}:=\{N<\overbar{W}<G\,\colon\, \overbar{W}/N\in\mathcal{D}\}$. By the Correspondence Theorem, $|\overbar{\mathcal{D}}|=|\mathcal{D}|$. We want to show that $\overbar{\mathcal{D}}$ is a dominating set for $\Gamma(G)$. Let $H$ be a proper non-trivial subgroup of $G$. If $H\cap N\neq 1$, then any element of $\overbar{\mathcal{D}}$ intersects $H$ non-trivially. Suppose that $H\cap N=1$. Take $W\in \mathcal{D}$ such that $(NH/N)\cap W:=Y\neq 1$. If $Y=NH/N$, then clearly $\overbar{W}$ contains $H$ where $\overbar{W}\in \overbar{\mathcal{D}}$ such that $\overbar{W}/N=W$. Suppose $Y\neq NH/N$ and let $\overbar{Y}/N=Y$. Obviously $\overbar{Y}<\overbar{W}$. Moreover, By the Dedekind's Lemma $\overbar{Y}=NK$, where $K:=\overbar{Y}\cap H$. That is $\overbar{W}\cap H\geq K\neq 1$. Since $H$ is an arbitrary subgroup, we see that $\overbar{\mathcal{D}}$ dominates $\Gamma(G)$. This proves the first part. The second part follows from the convention $\gamma(C_p)=\aleph_0$.
\end{proof}

\begin{cor}
  Let $G\cong H \times K$ be a finite group. Then $\gamma(G)\leq \mathrm{min}\{\gamma(H),\gamma(K)\}$. \qed
\end{cor}

Let $\mathcal{S}$ be a subset of the vertex set $V(G)$ of $\Gamma(G)$. It is natural to define the intersection graph $\Gamma(\mathcal{S})$ of the vertex set $\mathcal{S}$ in the following way. There is an edge between two vertices $H,K\in\mathcal{S}$ if and only if $H\cap K\in\mathcal{S}$. We denote the domination number of $\Gamma(\mathcal{S})$ simply by $\gamma(\mathcal{S})$. Let $V(G)_{>N}$ be the set of proper subgroups of $G$ containing the normal subgroup $N$ of $G$ strictly. By the Correspondence Theorem $\Gamma(V(G)_{>N})\cong\Gamma(G/N)$, hence $\gamma(V(G)_{>N})=\gamma(G/N)$. Moreover, by the proof of Lemma~\ref{lem:quotient}, a (non-empty) subset of $V(G)_{>N}$ dominates $\Gamma(G)$ provided that it dominates $\Gamma(V(G)_{>N})$. Let $\varphi_N\colon V(G)\cup\{1,G\}\to V(G)_{>N}\cup\{N,G\}$ be the map taking $H$ to $NH$; and let $\mathcal{D}$ be a dominating set of $\Gamma(G)$ consisting of maximal subgroups. Notice that even if $\varphi_N(\mathcal{D})\subseteq V(G)_{>N}$, the image set $\varphi_N(\mathcal{D})$ might not be a dominating set for $\Gamma(V(G)_{>N})$.

Let $\mathcal{S}_p(G)$ be the set of all proper non-trivial $p$-subgroups of $G$. Observe that if $G$ is a $p$-group, $\Gamma(G)$ and $\Gamma(\mathcal{S}_p(G))$ coincides. It is also true in general that there is an edge between two vertices of $\Gamma(\mathcal{S}_p(G))$ if and only if there is an edge between the corresponding vertices in $\Gamma(G)$.

 \begin{lem}
   Let $G$ be a finite group and let $N$ be a normal $p$-subgroup of $G$ such that the set $\mathcal{S}_p(G)_{>N}$ is non-empty. Then $$\gamma(\mathcal{S}_p(G))\leq\gamma(\mathcal{S}_p(G)_{>N}).$$
 \end{lem}

 \begin{proof}
   Let $\mathcal{D}$ be a dominating set of $\mathcal{S}_p(G)_{>N}$. As in the proof of Lemma~\ref{lem:quotient}, we want to show that $\mathcal{D}$ dominates $\Gamma(\mathcal{S}_p(G))$. Let $H\in\mathcal{S}_p(G)$. If $H\cap N\neq 1$, then clearly $H\cap W\neq 1$ for any $W\in\mathcal{D}$. Suppose $H\cap N=1$. Then $NH\in\mathcal{S}_p(G)_{>N}\cup\{G\}$. Let $Y$ be a minimal subgroup in $\mathcal{S}_p(G)_{>N}$ contained by $NH$; and let $W\in\mathcal{D}$ contains $Y$. If $Y=NH$, then $H<W$. Suppose $Y\neq NH$. Since $N<Y<NH$ and $N\cap H=1$, $Y\cap H\in \mathcal{S}_p(G)$ by the Product Formula; hence, $W\cap H\in \mathcal{S}_p(G)$ as well.

\end{proof}

\section{Abelian groups}
\label{sec:abelian}

In this section we classify finite abelian groups by their domination number. Recall that the exponent of a group $G$, denoted by $\mathrm{exp}(G)$, is the least common multiple of the orders of elements of $G$. Let $G$ be a finite group and consider the function $f$ from the set of non-empty subsets of $G$ to the set of positive integers taking $X\subseteq G$ to the lowest common multiple of the orders of elements of $X$. Clearly, the image of the whole group $G$ is $\mathrm{exp}(G)$. By a celebrated theorem of Frobenius if $X$ is a maximal subset of $G$ satisfying the condition $x^k=1$ for all $x\in X$ with $k$ is a fixed integer dividing $|G|$, then $k$ divides $|X|$. Let $g$ be the function taking the integer $k$ to the maximal subset $X_k:=\{x\in G\,\colon\, x^k=1\}$. Then, $f$ and $g$ define a Galois connection between the poset of non-empty subsets of $G$ and the poset of positive integers ordered by divisibility relation. In general such a maximal subset may not be a subgroup. For example, if the Sylow $p$-subgroup $P$ of $G$ is not a normal subgroup of $G$, then the union of conjugates of $P$ cannot be a subgroup. However, if $G$ is an abelian group then for any integer $k$ the subset $X_k$ is actually a subgroup.

For a finite group $G$ we denote by $\mathrm{sfp}(G)$ the square-free part of $|G|$, i.e. $\mathrm{sfp}(G)$ is the product of distinct primes dividing $|G|$. Notice that a collection of proper subgroups dominates to the intersection graph if and only if their union contains $X_t$, where $t=\mathrm{sfp}(G)$.
 
\begin{thm}\label{thm:abelian}
  Let $G$ be a finite abelian group with proper non-trivial subgroups. Then
  \begin{enumerate}
  \item $\gamma(G)=1$ if and only if $\mathrm{sfp}(G)<\mathrm{exp}(G)$.
  \item $\gamma(G)=2$ if and only if $\mathrm{sfp}(G)=\mathrm{exp}(G)$ and $\mathrm{sfp}(G)$ is not a prime number.
  \item $\gamma(G)=p+1$ if and only if $p=\mathrm{sfp}(G)=\mathrm{exp}(G)$ is a prime number.
  \end{enumerate}
\end{thm}

\begin{proof}
Let $t$ be the square-free part of $|G|$ and $m$ be the exponent of $G$.

\emph{Assertion 1.} Observe that $N_G=\{x\in G\,\colon\,  x^t=1\}$. By Lemma~\ref{lem:dom1}, $\gamma(G)=1$ if and only if $N_G$ is a proper subgroup which is the case if and only if $t<m$. 

For Assertion 2 and Assertion 3 it is enough to prove the sufficiency conditions, as $2\neq t+1$ for any prime number $t$.

\emph{Assertion 2.} Suppose that $t=m$ and $t$ is not a prime number. Then there exist two distinct primes $p$ and $q$ dividing $t$. Clearly, the subgroups $H=\{h\in G\,\colon\, h^{t/p}=1\}$ and $K=\{k\in G\,\colon\, k^{t/p}=1\}$ dominates $\Gamma(G)$. As there is no dominating set of cardinality one (by virtue of Assertion~1), $\gamma(G)=2$.

\emph{Assertion 3.} Suppose that $t=m$ and $t$ is a prime number. We consider $G$ as a vector space over the field $\mathbb{F}_t$ of $t$ elements of dimension $d\geq 2$ and fix a basis for this vector space in canonical way. Let $H_i=\langle h_i\rangle$, $1\leq i\leq t+1$ where $h_i=(1,0,\dots,0,i)$ for $i\leq t$ and $h_{t+1}=(0,\dots,0,1)$. Then $\{K_i:=H_i^{\perp}\,\colon\, 1\leq i\leq t+1\}$ is a dominating set for $\Gamma(G)$. To see this first observe that for any $g\in G$ there exists an $h_i$ such that $g\cdot h_i=0$, i.e. $g\in K_i$ for some $1\leq i\leq t+1$. Next, suppose that there exist a dominating set $\mathcal{D}=\{M_1,\dots,M_s\}$ with cardinality $s<t+1$. We want to derive a contradiction. Without loss of generality elements of $\mathcal{D}$ can be taken as maximal subgroups. Let $A_j=M_j^{\perp}$, $1\leq j\leq s$. Suppose $A_j$ are generated by linearly independent vectors (if not, we may take a maximal subset of $\{A_1,\dots,A_s\}$ with this property and apply the same arguments). By a change of basis if necessary, we may take $A_j=H_j$. However, that means $g=(1,0,\dots,0)$ is not contained in any $M_i\in\mathcal{D}$ as $g\cdot h_j\neq 0$ for $1\leq j\leq s$. This contradiction completes the proof.
\end{proof}

\begin{rem}
We may prove the first and second assertions by regarding $G$ as a $\mathbb{Z}$-module (compare with \cite[Theorem~4.4]{Yaraneri2013}). Also for the third assertion we may argue as follows. By Lemma~\ref{lem:quotient}, $\gamma(G)\leq t+1$ as $\gamma(C_p\times C_p)=t+1$ and the rank of $G$, say $r$, is greater than or equal to two. On the other hand, any non-identity element of $G$ belongs to exactly one minimal subgroup and there are $t^r-1$ of them. Any maximal subgroup of $G$ contains $t^{r-1}-1$ non-identity elements and $t$ maximal subgroups may cover at most $t^r-t$ elements; hence, there is no dominating set for $G$ of size $< t+1$.
\end{rem}

By the fundamental theorem of finite abelian groups any finite abelian group can be written as the direct product of cyclic groups of prime power orders, thus we may restate Theorem~\ref{thm:abelian} as
$$
\gamma(C_{p_1^{\alpha_1}}\times\cdots\times C_{p_1^{\alpha_k}})=
\begin{cases}
  1 & \text{if } \alpha_i\geq 2 \text{ for some } 1\leq i\leq k \\
  2 & \text{if } \alpha_i=1 \text{ for all } 1\leq i\leq k \\ 
    & \text{ and } p_{j_1}\neq p_{j_2} \text{ for some } j_1\neq j_2 \\
p+1 & \text{if } \alpha_i=1 \text{ and } p_i=p \text{ for all } 1\leq i\leq k
\end{cases}
$$
where $p_1,\ldots,p_k$ are prime numbers and  $\alpha_1,\ldots,\alpha_k$ are positive integers with $k=1$ implies $\alpha_1>1$.

\section{Solvable groups}
\label{sec:solvable}

Though finite abelian groups can be classified by their domination number, it seems this is not possible in general. Nevertheless, we may use the structural results to find upper bounds for the domination number of groups belonging to larger families. 

\begin{prop}
  Let $G$ be a finite nilpotent group and $p$ be a prime number. Suppose $G\ncong C_p$. Then 
\begin{enumerate}
  \item $\gamma(G)\leq p+1$ if $G$ is a $p$-group.
  \item $\gamma(G)\leq 2$ if $G$ is not a $p$-group.
\end{enumerate}  
\end{prop}

\begin{proof}
  \emph{Assertion 1.} There is a normal subgroup $N$ of $G$ of index $p^2$ such that the quotient group $G/N$ is isomorphic to either $C_{p^2}$ or $C_p\times C_p$. The assertion follows from Lemma~\ref{lem:quotient}.
  
  \emph{Assertion 2.} Let $G$ be the internal direct product of $P$, $Q$ and $N$, where $P$ and $Q$ are the non-trivial Sylow $p$- and Sylow $q$- subgroups of $G$. Clearly, $NP$ and $NQ$ form a dominating set of $\Gamma(G)$.
\end{proof}

Let $G$ be a finite group. We denote by $R_G$ the intersection of the subgroups in the lower central series of $G$. This subgroup is the smallest subgroup of $G$ in which the quotient group $G/R_G$ is nilpotent. Obviously, the nilpotent residual $R_G$ is a proper subgroup whenever $G$ is a solvable group.

\begin{cor}\label{cor:nilpotent}
Let $G$ be a finite group such that $G/R_G$ has proper non-trivial subgroups. Then 
\begin{enumerate}
  \item $\gamma(G)\leq p+1$ if $G/R_G$ is a $p$-group.
  \item $\gamma(G)\leq 2$ if $G/R_G$ is not a $p$-group.
\end{enumerate}  \qed
\end{cor}

Let $D_{2n}$ denotes the dihedral group of order $2n$. Then $R_{D_{2n}}=D_{2n}'\cong C_n$ and so the quotient $D_{2n}/R_{D_{2n}}$ has no proper non-trivial subgroups; thus, Corollary~\ref{cor:nilpotent} does not apply in this case. Nevertheless, the structure of dihedral groups are fairly specific allowing us to determine exact formulas for their domination number depending on the order $2n$.

\begin{lem}\label{lem:dihedral}
    Let $p$ be the smallest prime dividing $n$. Then $$ \gamma(D_{2n})=
    \begin{cases}
      p\,, & \text{if }\, p^2\divides  n \\
      p+1\,, & \text{otherwise}
    \end{cases} $$
  \end{lem}
  
\begin{proof}
Let $D_{2n}:=\langle a,b\divides  a^n=b^2=1, bab=a^{-1}\rangle$. Elements of $D_{2n}$ can be listed as $$ \{1,a,a^2,\dots,a^{n-1},b,ab,a^2b,\dots,a^{n-1}b\}.  $$ It can be easily observed that the latter half of those elements are all of order two; and hence, the minimal subgroups of $D_{2n}$ consists of subgroups of $\langle a\rangle$ that are of prime order and subgroups $\langle a^jb\rangle$, where $j\in\{0,1,\dots,n-1\}$. And the maximal subgroups of $D_{2n}$ consists of $\langle a\rangle$ and subgroups of the form $\langle a^t,a^rb\rangle$, where $t\divides n$ is a prime number. Fix a prime $t\divides n$. Observe that any element of the form $a^jb$ is contained exactly one of the maximal subgroups $T_r:=\langle a^t,a^rb\rangle$ with $r\in\{0,1,\dots,t-1\}$. Let $p$ be the smallest prime dividing $n$. Obviously, the union of the subgroups $P_r:=\langle a^p,a^rb\rangle$ with $r\in\{0,1,\dots,p-1\}$ contains all minimal subgroups of the form $\langle a^jb\rangle$, $j\in\{0,1,\dots,n-1\}$, and there are no possible way to cover them with fewer than $p$ subgroups. Finally, if $p^2\divides n$, those subgroups contains all minimal subgroups of $\langle a\rangle$; and if $p^2\ndivides n$, we must take $P_r$, $r\in\{0,1,\ldots,p-1\}$, together with a subgroup containing $\langle a^{n/p}\rangle$ to form a dominating set with least cardinality.
\end{proof}

\begin{rem}\label{rem:nsum}
  It is difficult to find an $n$-sum group $G$ such that $\gamma(G)=n$ and $\mathrm{sfp}(G)<\mathrm{exp}(G)$. One example satisfying those conditions is the dihedral group of order $36$. By Lemma~\ref{lem:dihedral}, $\gamma(D_{36})=3$. Moreover, any cyclic subgroup of $D_{36}$ is contained by those three subgroups that are of index two; and hence, $D_{36}$ is a $3$-sum group. 
\end{rem}

Consider the \emph{normal} series of subgroups
$$ G=G_0\geq G_1\geq \dots \geq G_k=1.  $$
By the third isomorphism theorem and Lemma~\ref{lem:quotient}, $\gamma(\frac{G_0/G_2}{G_1/G_2})=\gamma(G_0/G_1)\geq\gamma(G_0/G_2)$. And, by a repeated application, we have
$$ \gamma(G)\leq \gamma(G/G_{k-1})\leq\dots \leq \gamma(G/G_1).  $$

\begin{prop}\label{prop:super}
  Let $G$ be a finite supersolvable group with proper non-trivial subgroups. Then $\gamma(G)\leq p+1$ for some prime divisor $p$ of $|G|$.
\end{prop}

\begin{proof}
  Since $G$ is a finite supersolvable group, it has a normal subgroup $N$ of index $m$ where $m$ is either a prime square or a product of two distinct primes. In the first case, $\gamma(G/N)$ is at most $p+1$ where $p=\sqrt{m}$, and in the latter case $\gamma(G/N)$ is at most $2$. The proof follows from Lemma~\ref{lem:quotient}.
\end{proof}

At this stage it is tempting to conjecture that the assertion of Proposition~\ref{prop:super} holds more generally for solvable groups. However, we may construct a counterexample in the following way. Let $G=NH$ be a Frobenius group with complement $H$ such that the kernel $N$ is a minimal normal subgroup of $G$. Further, suppose that $H\cong C_q$ for some prime $q$. Notice that since Frobenius kernels are nilpotent, $G$ must be a solvable group. On the other hand, since $N$ is a minimal normal subgroup, $N$ has no characteristic subgroup and in particular $N\cong C_p\times\dots\times C_p$ for some prime $p$. Suppose that the rank $r$ of $N$ is greater than $1$. Now, since $N_G(H)=H$ and $N$ is a minimal normal subgroup, each conjugate of $H$ is a maximal subgroup. That means there are totally $|G:H|=p^r$ isolated vertices in $\Gamma(G)$ which, in turn, implies that $\gamma(G)=p^r+1$. As $p^r=1+kq$ for some integer $k$ by Sylow's Theorem, the domination number $\gamma(G)$ is greater than both $p+1$ and $q+1$.

\begin{prop}\label{prop:solvable}
  Let $G$ be a finite solvable group and $H,K$ be a pair of maximal subgroups such that $(|G:H|,|G:K|)=1$. Then, $\gamma(G)\leq |G:N_G(H)| + |G:N_G(K)|$.
\end{prop}

Notice that Hall's Theorem guarantees the existence of such a pair of maximal subgroups provided $G$ is not a $p$-group. Notice also that one of the summands in the stated inequality can always be taken $1$.

\begin{proof}
  Take a minimal subgroup $A$. If $|A|\divides |G:H|$, then by Hall's Theorem a conjugate of $K$ contains $A$. And if $|A|\ndivides |G:H|$, then clearly a conjugate of $H$ contains $A$.
\end{proof}

\section{Permutation groups}
\label{sec:permutation}

We denote by $S_X$ the group of the permutations of the elements of $X$. If $X= [n] := \{1,2,\dots,n\}$, we simply write $S_n$; and $A_n$ will be the group of even permutations on $n$ letters. Notice that $S_n/R_{S_n}=S_n/A_n\cong C_2$.

\begin{lem}\label{lem:SnAn}
  $\gamma(S_n)\neq 1$ and $\gamma(A_n)\neq 1$.   
\end{lem}

\begin{proof}
  As $S_2\cong C_2$, we see that $\gamma(S_2)\neq 1$. For $n\geq 3$, since any transposition generates a minimal subgroup of the symmetric group and since adjacent transpositions generate the whole group, $\gamma(S_n)\neq 1$ by Lemma~\ref{lem:dom1}. Similarly, $\gamma(A_3)\neq 1$; and for $n\geq 4$, since $3$-cycles generate $A_n$, $\gamma(A_n)\neq 1$ by the same lemma.
\end{proof}

Let $G$ be a permutation group acting faithfully on $X=\{1,2,\ldots,n\}$. An element $g\in G$ is called \emph{homogeneous} if the associated permutation has cycle type $(p^{k},1^{n-pk})$ with $p$ a prime. Since any minimal subgroup of $G$ is a cyclic group of prime order, in order for $\mathcal{D}\subseteq V(G)$ to be a dominating set of $\Gamma(G)$, the union of the elements of $\mathcal{D}$ must contain every homogeneous element of $G$ and vice versa.

\begin{thm}\label{thm:Sn}
Let $\vartheta\colon \mathbb{N}\to\mathbb{N}$ be the function given by
$$ \vartheta(n)=\left\{
   \begin{array}{llc}
      n+1\,, & \text{if }\, n=2k+1 &  \\
      n\,, & \text{if }\, n=2k & 
   \end{array}\right.
$$
Then $\gamma(S_n)\leq \vartheta(n)$.
\end{thm}

\begin{proof}
Let $\mathcal{D}$ be a dominating set for $\Gamma(S_n)$ of minimal size. Without loss of generality we may assume that the elements of $\mathcal{D}$ are maximal subgroups of $S_n$. We know that every homogeneous element of $S_n$ must belong to some subgroup in $\mathcal{D}$, in particular, every involution (elements of order 2) must be covered by $\mathcal{D}$. Observe that if some collection of subgroups covers all involutions of type   $(2^l,1^{n-2l})$, with $l$ is odd, but not all homogeneous elements, then by adding $A_n$ to this collection we would get a dominating set for $\Gamma(S_n)$, for $A_n$ contains all cycles of prime length $>2$. 

For $n=2k+1$, any involution must fix a point, hence must belong to some point stabilizer. Hence, together with $A_n$ we have a dominating set of cardinality $n+1$.

For $n=2k$, we consider the $2$-set stabilizers of $S_n$. Let $X_j=\{1,j\}$ and $Y_j=[n]\setminus X_j$, where $2\leq j\leq n$. Observe that the union of the subgroups $S_{X_j}\times S_{Y_j}$ of $S_n$ contains all involutions. Thus, together with $A_n$ we have a dominating set of cardinality $n$.
\end{proof}

\begin{cor}\label{cor:primitive}
  Let $G$ be a primitive subgroup of $S_n$ containing an odd permutation, and let $\vartheta$ be the function defined in Theorem~\ref{thm:Sn}. Then, the inequality $\gamma(G)\leq \vartheta(n)$ holds. 
\end{cor}

\begin{proof}
  Since $G$ is a primitive subgroup, it is not contained by any imprimitive subgroup. The proof follows from the fact that the only primitive subgroup used to form dominating sets in the proof of Theorem~\ref{thm:Sn} is $A_n$, and from Lemma~\ref{lem:subgroup}.
\end{proof}

\section{Intersection graphs of $G$-sets}
\label{sec:G-sets}

As is well-known any transitive $G$-set $\Omega$ is equivalent to a (left) coset space $G/G_x$, where $G_x$ is the stabilizer of a point $x\in\Omega$. So, for example if $\Omega$ is a regular $G$-set, then it can be represented with $G/1$; and if $\Omega$ is the trivial $G$-set, it is represented with $G/G$. Since any $G$-set is the disjoint union of transitive $G$-sets, given a $G$-set $\Omega$ it can be represented as the sum of coset spaces. For a subgroup $H$ of $G$ we denote by $(H)$ the conjugacy class of $H$ in $G$, and by $[G/H]$ the isomorphism class of the transitive $G$-set $G/H$. It is well-known that $[G/H]=[G/K]$ if and only if $H=\leftidx{^g}{K}$ for some $g\in G$.

The \emph{Burnside ring} $B(G)$ of $G$ is the ring generated by the isomorphism classes of $G$-sets, where addition is the disjoint union and product is the Cartesian product of $G$-sets. Therefore a typical element of $B(G)$ is of the form
$$
\sum\limits_j a_j[G/H_j]
$$  
where $a_j$ are integers and $H_j$ are the representatives of the conjugacy classes of subgroups of $G$. Let $A$ and $B$ be normal subgroups of a group $G$. Then the canonical map $gA\cap B\mapsto (gA,gB)$ from $G/(A\cap B)$ to $G/A\times G/B$ is an injective group homomorphism. Now, let us consider two arbitrary subgroups (not
necessarily normal) $H_1,H_2$ of $G$. In this case, $G/H_1$ and $G/H_2$ are still $G$-sets and the diagonal action of $G$ on the Cartesian product $G/H_1\times G/H_2$ yields the map
\begin{align*}
  \phi \colon G/(H_1\cap H_2) & \longrightarrow\; G/H_1\times G/H_2 \\
        gH_1\cap H_2 & \longmapsto\; (gH_1,gH_2)
\end{align*}
which is an injective $G$-equivariant map. That is to say, subgroup
intersection is related with the multiplication operator of the Burnside ring.

Let $\Omega_1$ and $\Omega_2$ be two $G$-sets and let $\mathcal{R}_{(\Omega_1,\Omega_2)}$ be a set of representative elements for the orbits of $\Omega_1\times\Omega_2$. The Cartesian product $\Omega_1\times\Omega_2$ decomposes into the disjoint union of transitive $G$-sets in a non-trivial fashion. More precisely,
$$ \Omega_1\times\Omega_2 \cong \bigsqcup\limits_{(x,y)\in\mathcal{R}_{(\Omega_1,\Omega_2)}} G/(G_x\cap G_y) $$
where $G_x$ and $G_y$ are the stabilizers of the points $x\in\Omega_1$ and $y\in\Omega_2$ respectively. Let $\mathcal{R}_{(H,K)}$ be a set of representatives for the equivalence classes of $(H,K)$-double cosets. Setting $\Omega_1=G/H$, $\Omega_2=G/K$ and using sigma notation for disjoint union, we may write
$$[G/H][G/K]=\sum\limits_{g\in\mathcal{R}_{(H,K)}}[G/(H\cap \leftidx{^g}{K})]$$
as the orbits of $G/H\times G/K$ are parametrized by the $(H,K)$-double cosets. For more information the reader may refer to \cite[Chapter~I]{tomDieck1979}. 

Let $G$ be a fixed finite group. The following definition is suggested to me by Ergün Yaraneri. The \emph{intersection graph} $\Gamma[\Omega]$ of a $G$-set $\Omega$ is the simple graph with vertex set the proper non-trivial stabilizers of points in $\Omega$ and there is an edge between two distinct stabilizers if and only if their intersection is non-trivial. We used in this notation ``brackets'' instead of ``parentheses'' to emphasize that the argument is a $G$-set. Observe that since $G$ is a finite group, there are at most finitely many intersection graphs $\Gamma[\Omega]$ that can be associated with $G$-sets. Notice that both $\Gamma[G/1]$ and $\Gamma[G/G]$ are empty graphs by definition.

\begin{example}
  Take $\Omega:=\{1,2,\dots,n\}$ as the vertex set of a regular $n$-gon on plane with $n$ being an odd number. Since the automorphism group of this polygon is isomorphic to the dihedral group $D_{2n}$, considered as a $D_{2n}$-set a stabilizer of a point of $\Omega$ is corresponding to a unique involution of $D_{2n}$. Those involutions form a single conjugacy class and $\Gamma[\Omega]$ consists of $n$ isolated vertices. Let $G$ be a finite group and $H$ be a subgroup of $G$. As a general fact $\Gamma[G/H]$ consists of $|G:H|$ isolated vertices if and only if $G$ is a Frobenius group with complement $H$.
\end{example}

We denote by $\mathcal{C}(G)$ the set of conjugacy classes of subgroups of $G$. By the following Proposition, intersection graphs of groups can be seen as particular cases of intersection graphs of $G$-sets.

\begin{prop}\label{prop:sigma}
  Let $G$ be a finite group and  $\Sigma:=\bigsqcup_{(H)\in\mathcal{C}(G)}[G/H]$. Then
  $$ \Gamma[\Sigma]=\Gamma(G).$$
\end{prop}

\begin{proof}
  Since there is an edge between two vertices of $\Gamma[\Sigma]$ if and only if their intersection is non-trivial, it is enough to show that the vertex set of $\Gamma[\Sigma]$ is $V(G)$. However, this is obvious as the set of subgroups of $G$ can be partitioned into the conjugacy classes of subgroups:
  $$ V(G)\cup\{1, G\}=\{\leftidx{^g}{H}\,\colon\,(H)\in\mathcal{C}(G),\; g\in G\}.$$
 Notice that no vertices associated with $1$ and $G$ are in $\Gamma[\Sigma]$ by its definition.
\end{proof}

Let $G$ be a finite group. We denote by $\mathcal{A}(G)$ the set of minimal subgroups of $G$ and by $\mathcal{M}(G)$ the set of maximal subgroups. The following characterizations are easy to deduce:
\begin{itemize}
\item $N\triangleleft G \iff [G/K][G/N]=|G:NK|[G/(N\cap K)]$ for every $K\leq G$.
\item $A\in\mathcal{A}(G)\cup \{1\} \iff $ for any $K\leq G$ there are non-negative integers $a$ and $b$ such that $[G/K][G/A]=a[G/A] + b[G/1]$.  
\item $H\in\mathcal{M}(G) \iff [G/K][G/H]$ does \emph{not} contain $[G/H]$ as a summand unless $K=G$ or $K\in (H)$.
\end{itemize}
Let $A$ be a minimal subgroup, and $H$ be a maximal subgroup of $G$. Observe that in case $G$ is abelian the following equality holds.
$$
[G/H][G/A]=
\begin{cases}
  |G:H|[G/A] & \text{if } A\leq H \\
  [G/1] & \text{if } A\nleq H
\end{cases}
$$
We may restate Theorem~\ref{thm:abelian} in the following way:
\begin{enumerate}
   \item $\gamma(G)=1$ if and only if ($\ast$) there exists an $H\in\mathcal{M}(G)$ such that for every $A\in\mathcal{A}(G)$ the equality $[G/H][G/A] = |G:H|[G/A]$ holds.
   \item $\gamma(G)=2$ if and only if ($\ast\ast$) for every $H\in\mathcal{M}(G)$ there exists an $A\in\mathcal{A}(G)$ such that $|G:H|\ndivides |[G/A]|$.
   \item $\gamma(G)=p+1$ if and only if neither ($\ast$) nor ($\ast\ast$) holds and $p=|G:H|=|A|$, where $H\in\mathcal{M}(G)$ and $A\in\mathcal{A}(G)$.
\end{enumerate}

\begin{prop}\label{prop:index}
  Let $G$ be a finite group and let $\mathcal{H}=\{H_i\in V(G)\,\colon\, 1\leq i\leq s\}$ be a set having the property that for any $K\in V(G)$ there exists an $H_i\in\mathcal{H}$ such that $[G/K][G/H_i]\neq k[G/1]$ for every positive integer $k$. Then $$\gamma(G)\leq\sum\limits_{i=1}^s|G:N_G(H_i)|.$$ Moreover, $\gamma(G)=1$ if and only if there exists $H\triangleleft G$ such that $[G/K][G/H]\neq k[G/1]$ for every $K\in V(G)$. 
\end{prop}

\begin{proof}
  Let $\mathcal{H}$ be a set as in the statement of the Proposition. We want to show that $\mathcal{D}:=\{\leftidx{^g}{H_i}\,\colon\, H_i\in\mathcal{H},\, g\in G\}$ is a dominating set for $\Gamma(G)$. However this is obvious, since for any $K\in V(G)$, $[G/K][G/H_i]\neq k[G/1]$ implies that $K\cap\leftidx{^g}{H_i}\neq 1$ for some $g\in G$. This completes the first part of the proof.

Now, suppose that $\gamma(G)=1$. Let $L<G$ such that $L\cap K\neq 1$ for every $K\in V(G)$. Obviously, $[G/K][G/L]=\sum [G/(K\cap\leftidx{^g}{L})]=\sum [G/(\leftidx{^{g^{-1}}}{K}\cap L)]$ does not contain a regular summand, as $L$ intersects every conjugate of $K$ non-trivially. Take a conjugate $\leftidx{^a}{L}$ of $L$. Since $\leftidx{^a}{L}$ also forms a dominating set, $[G/K][G/(L\cap\leftidx{^a}{L})]=\sum [G/(\leftidx{^{g^{-1}}}{K}\cap L\cap\leftidx{^a}{L})]$ does not contain a regular summand either. Let $H:=\bigcap\leftidx{^g}{L}$. By a repeated application of the above argument we see that $[G/K][G/H]$ has no regular summand. However, $H$ is a normal subgroup which means $[G/K][G/H]\neq k[G/1]$.  

Conversely, suppose that there exists an $H\triangleleft G$ such that $[G/K][G/H]\neq k[G/1]$ for every $K\in V(G)$. Since $H$ is normal, this means $[G/K][G/H]$ has no regular summand which, in turn, implies that $\{H\}$ is a dominating set for $\Gamma(G)$. This completes the proof. 
\end{proof}

\begin{rem}
Clearly, the elements of the set $\mathcal{H}$ in the statement of Proposition~\ref{prop:index} can be taken as maximal subgroups. Then, Proposition~\ref{prop:solvable} tells us that there exists such a $2$-element set if $G$ is a solvable group but \emph{not} a $p$-group. 
\end{rem}

\section{Intersection complexes}
\label{sec:complex}

Recall that a \emph{(abstract) simplicial complex} $\mathcal{S}$ over a set $X$ is a finite collection of subsets of $X$ such that the union of those subsets is $X$ and if $\sigma$ is an element of $\mathcal{S}$, so is every subset of $\sigma$. The element $\sigma$ of $\mathcal{S}$ is called a simplex of $\mathcal{S}$ and each subset of $\sigma$ is called a face of $\sigma$. The \emph{$k$-skeleton} of $\mathcal{S}$ is the subcollection of elements of $\mathcal{S}$ having cardinality at most $k+1$; hence, the $0$-skeleton of $\mathcal{S}$ is the underlying set $X$ plus the empty set. For a group $G$, we define the \emph{intersection complex $K(G)$} of $G$ as the simplicial complex whose faces are the sets of proper subgroups of $G$ which intersect non-trivially. As a graph the $1$-skeleton of $K(G)$ is isomorphic to the intersection graph $\Gamma(G)$. This notion can be compared with the two other notions in literature, namely the order complex and the clique complex. In the first case, we begin with a poset (poset of proper non-trivial subgroups of a group in our case) and construct its order complex by declaring chains of the poset as the simplices. And, in the latter case, we take a graph (i.e. the intersection graph of a group) and the underlying set of the corresponding clique complex is the vertex set of the graph with simplices being the cliques.

\begin{examples}
\mbox{}
  \begin{enumerate} 
  \item \label{example:Q8} The quaternion group $Q_8$ has three maximal subgroups, say $\langle i\rangle,\langle j\rangle$, and $\langle k\rangle$, of order four intersecting at the unique minimal subgroup $\{-1,1\}$. Thus, $\Gamma(Q_8)$ is a complete graph $K_4$ depicted in Figure~\ref{fig:Q8}. Moreover, $K(Q_8)$ is a tetrahedron as those four vertices form a simplex. However, the order complex of the poset of proper non-trivial subgroups of $Q_8$ is isomorphic to the star graph $K_{1,3}$ as a graph. Hence, order complexes and intersection complexes are different in general. Notice that the clique complex of $\Gamma(Q_8)$ is the same as $K(Q_8)$.
  \item \label{example:C2}The intersection graph of the elementary abelian group of order eight  is represented in Figure~\ref{fig:C2}. Here the vertices on the outer circle represents the minimal subgroups and the vertices on the inner circle are the maximal subgroups. By the Product Formula any two maximal subgroups intersects at a subgroup of order $2$. Therefore, the vertices in the inner circle form a complete subgraph and those vertices form a simplex in the clique complex whereas they do not in $K(C_2\times C_2\times C_2)$. Thus, intersection complexes and clique complexes are not the same in general. Notice that  $\Gamma(C_2\times C_2\times C_2)$ is symmetrical enough to reflect the vector space structure of the group.
  \end{enumerate}
\end{examples}

\begin{figure}[h]
\centering
\begin{subfigure}{.5\textwidth}
  \centering
  \includegraphics[width=.5\textwidth]{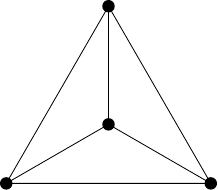}
  \caption{$\Gamma(Q_8)$}
  \label{fig:Q8}
\end{subfigure}%
\begin{subfigure}{.5\textwidth}
  \centering
  \includegraphics[width=.6\textwidth]{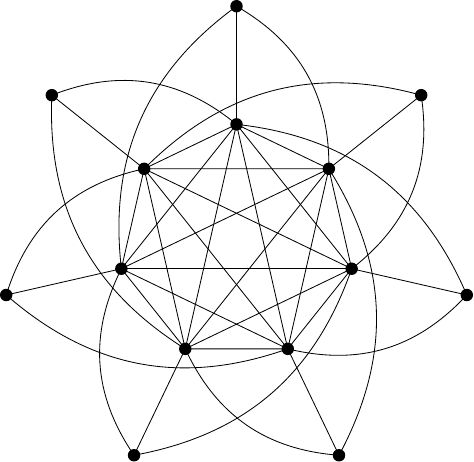}
  \caption{$\Gamma(C_2\times C_2\times C_2)$}
  \label{fig:C2}
\end{subfigure}
\caption{Intersection graphs of some groups of order $8$}
\label{fig:Q8C2}
\end{figure}

An important result on the subject for our purposes (see Lemma~\ref{lem:nerve} below) uses the definitions of algebraic topology adapted to `poset' context. By a \emph{covering} $C$ of a finite poset $\mathcal{P}$ we mean a finite collection $\{C_i\}_{i\in I}$ of subsets of $\mathcal{P}$, such that $\mathcal{P}=\cup_{i\in I}C_i$. The \emph{nerve} $\mathcal{N}(C)$ of $C$ is the simplicial complex whose underlying set is $I$ and the non-empty simplices are the $J\subseteq I$ such that $C_J:=\cap_{i\in J}C_i\neq\varnothing$. The covering $C$ is called \emph{contractible} if each $C_J$ is contractible considered as an order complex, where $J$ is a simplex in $\mathcal{N}(C)$. We say $C$ is a \emph{downward closed} covering if each $C_i$, $i\in I$, is a closed subset of $\mathcal{P}$, i.e. for each $C_i$ if the condition
$$ \text{whenever } x\in C_i \text{ and } x'\leq_{\mathcal{P}} x, \text{ then } x'\in C_i $$
is satisfied. Of course, we may define \emph{upward closed} coverings dually.

\begin{lem}[see {\cite[Theorem~4.5.2]{Smith2011}}]\label{lem:nerve}
  If $C$ is a (upward or downward) closed contractible covering of a poset $\mathcal{P}$, then the order complex of $\mathcal{P}$ is homotopy equivalent to nerve $\mathcal{N}(C)$ of $C$.
\end{lem}

Let $x$ be an element of a poset $\mathcal{P}$. We denote by $\mathcal{P}_{\leq x}$ the set of elements $x'$ of $\mathcal{P}$ satisfying $x'\leq_{\mathcal{P}} x$. Similarly, $\mathcal{P}_{\geq x} := \{x'\in \mathcal{P}\colon x'\geq x\}$. The sets $\mathcal{P}_{\leq x}$ and $\mathcal{P}_{\geq x}$ are called \emph{cones} in poset terminology, and it is a standard fact that they are contractible. Consider the set
$$
C := \{\mathcal{P}_{\geq x}\,\colon\, x \text{ is a minimal element in the ordering of } \mathcal{P}\}.
$$
Clearly, $C$ is a contractible upward closed covering of $\mathcal{P}$; hence, by Lemma~\ref{lem:nerve} the nerve $\mathcal{N}(C)$ of $C$ is homotopy equivalent to the order complex of $\mathcal{P}$. In Example~\ref{example:Q8}, we remarked that order complexes and intersection complexes are different in general. However, they are equivalent up to homotopy. 

\begin{prop}\label{prop:homotopy}
  For a group $G$ the intersection complex $K(G)$ and the order complex of the poset of proper non-trivial subgroups of $G$ are homotopy equivalent. 
\end{prop}

\begin{proof}
  Consider the face poset $\mathcal{F}$ of $K(G)$, i.e. the poset of simplices ordered by inclusion. For a proper non-trivial subgroup $H$ of $G$, let $C_H$ be the subset $\mathcal{F}_{\geq \{H\}}$ of $\mathcal{F}$. Then the collection
$$
C := \{C_H\,\colon\, 1 < H < G\}
$$
is an upward closed contractible covering of $\mathcal{F}$, as the singletons $\{H\}$ are exactly the minimal elements of $\mathcal{F}$. Therefore, the order complex of $\mathcal{F}$ and the nerve of $C$ are of the same homotopy type by Lemma~\ref{lem:nerve}. Observe that $\mathcal{N}(C)$ is exactly the intersection complex $K(G)$ of $G$. Since the order complex of $\mathcal{F}$ is the barycentric subdivision of $K(G)$, we see that $K(G)$ is homotopy equivalent to the order complex of the poset of proper non-trivial subgroups of $G$.
\end{proof}

\begin{remarks}
  \mbox{}
  \begin{enumerate}
  \item Intersection complexes can be considered as a special instance of a more general construction in which given a poset $\mathcal{P}$ we form a simplicial complex $K(\mathcal{P})$ by declaring simplices as the subsets of the poset having a well-defined meet. It is easy to see that the above proof can be adapted to work in this frame. Recall that in Section~\ref{sec:pre} we defined $\mathcal{S}_p(G)$ as the set of proper non-trivial $p$-subgroups of $G$. Considered as a poset order complex of $\mathcal{S}_p(G)$ shares the same homotopy type with $K(\mathcal{S}_p(G))$.
  \item An alternative argument to prove Proposition~\ref{prop:homotopy} which is due to Volkmar Welker is as follows: Consider the face poset $\mathcal{F}$ of $K(G)$. By the identification $H\mapsto \sigma_H$, the poset of proper non-trivial subgroups of $G$ becomes a subposet (after reversing the order relation) of $\mathcal{F}$. We want to show that $\mathcal{F}$ and the poset of the proper non-trivial subgroups of $G$ are of the same homotopy type as order complexes. Let $f$ be the map taking a simplex $\sigma$ in $\mathcal{F}$ to $\sigma_K$, where $K$ is the intersection of all maximal subgroups containing the intersection of all the elements in $\sigma$ as subgroups. Then $f$ is a closure operator on $\mathcal{F}$. Let $g$ be the map taking $H$ to $K$, where $K$ is the intersection of all maximal subgroups containing $H$. Then $g$ is a closure operator on the poset of proper non-trivial subgroups of $G$. Since closure operations on posets preserve the homotopy type of the order complex and since the images of $f$ and $g$ are isomorphic by the identification $K\mapsto \sigma_K$, the proof is completed.
  \end{enumerate}
\end{remarks}

Let $G$ be a finite group. For a proper non-trivial subgroup $H$ of $G$, we denote by $V(G)_{\geq H}$ the set of proper subgroups of $G$ containing $H$. Similarly, $V(G)_{\leq H}$ is the set non-trivial subgroups of $G$ contained by $H$. Consider the following collections:
\begin{align*}
  A & := \{V(G)_{\geq H}\,\colon\, H \text{ is a minimal subgroup of } G\} \\
  M & := \{V(G)_{\leq H}\,\colon\, H \text{ is a maximal subgroup of } G\}
\end{align*}
Since $A$ is an upward closed contractible covering of $V(G)$ considered as a poset under set theoretical inclusion and $M$ is a downward closed contractible covering, by Lemma~\ref{lem:nerve} $\mathcal{N}(A), \mathcal{N}(M)$, and the order complex of proper non-trivial subgroups of $G$ share the same homotopy type. Recall that the Frattini subgroup $\Phi(G)$ of $G$ is the intersection of all maximal subgroups of $G$.

\begin{thm}\label{thm:FratDom}
  Let $G$ be a finite group. Then 
  \begin{enumerate}[(i)]
  \item $\Phi(G)\neq 1$ if and only if $\mathcal{N}(M)$ is a simplex. 
  \item $\gamma(G)=1$ if and only if $\mathcal{N}(A)$ is a simplex.
  \end{enumerate}
\end{thm}

\begin{proof}
  \emph{Assertion 1.} Clearly, $\mathcal{N}(M)$ is a simplex if and only if for any subset $\mathcal{S}$ of $\mathcal{N}(M)$ intersection of the subgroups in $\mathcal{S}$ is non-trivial which is the case if and only if intersection of all maximal subgroups of $G$ is non-trivial.

  \emph{Assertion 2.} Similar to the previous case $\mathcal{N}(A)$ is a simplex if and only if $N_G$ is a proper subgroup of $G$. The assertion follows from Lemma~\ref{lem:dom1}.
\end{proof}

\begin{rem}
  Using Theorem~\ref{thm:abelian}, one may conclude that for abelian groups $\gamma(G)=1$ \emph{if and only if} $\Phi(G)\neq 1$. Let $P$ be a $p$-group with $p$ being a prime number. It is also true that $\gamma(G)=1$ implies $\Phi(G)\neq 1$; for if $P$ is a non-cyclic $p$-group and $\gamma(P)=1$, then $P/\Phi(P)$ is elementary abelian of rank $>1$ and $\gamma(P/\Phi(P))=p+1$ in that case. However, the converse statement is not true even for $p$-groups. For example, $\Phi(D_8)\cong C_2$ but $\gamma(D_8)=2$.
\end{rem}

Since $\mathcal{N}(A)$ and $K(G)$ are of the same homotopy type by Proposition~\ref{prop:homotopy} and Lemma~\ref{lem:nerve}, as a consequence of Theorem~\ref{thm:FratDom} we have the following

\begin{cor}\label{cor:dom1}
Let $G$ be a finite group. If $\gamma(G)=1$, then $K(G)$ is contractible.  
\end{cor}

\section{Acknowledgments}

We thank to the anonymous referee for the invaluable comments and suggestions which greatly enhances the exposition of the paper. In particular, the proof idea of Proposition~\ref{prop:homotopy} is due to the referee.


\newcommand{\etalchar}[1]{$^{#1}$}
\providecommand{\bysame}{\leavevmode\hbox to3em{\hrulefill}\thinspace}
\providecommand{\MR}{\relax\ifhmode\unskip\space\fi MR }
\providecommand{\MRhref}[2]{%
  \href{http://www.ams.org/mathscinet-getitem?mr=#1}{#2}
}
\providecommand{\href}[2]{#2}

\end{document}